\documentclass[12pt,leqno]{article}
\usepackage{amsfonts}
\usepackage{amsmath, amsthm, amsfonts, amssymb, color}
\usepackage{mathrsfs}
\renewcommand{\bar}{\overline}
\renewcommand{\hat}{\widehat}
\renewcommand{\tilde}{\widetilde}

\newtheorem{thm}{Theorem}[section]

\newtheorem{lem}[thm]{Lemma}

\theoremstyle{definition}

\newcommand{\scr}[1]{\mathscr #1}
\definecolor{wco}{rgb}{0.5,0.2,0.3}

\numberwithin{equation}{section} \theoremstyle{remark}
\newtheorem{rem}{Remark}[section]

\newcommand{\ua}{\uparrow}

\title{{\bf Large deviations for neutral stochastic functional differential equations}
}
\author{
{\bf  Yongqiang Suo and Chenggui Yuan
 }\\
\footnotesize{Department of Mathematics, Swansea University, Bay Campus, SA1 8EN, UK}\\
}

\begin{document}
\def\A{\mathscr{A}}
\def\G{\mathscr{G}}
\def\eq{\equation}
\def\bg{\begin}
\def\ep{\epsilon}
\def\111{±ßÖµÎÊÌâ(1)--(2)}
\def\x{\|x\|}
\def\y{\|y\|}
\def\xr{\|x\|_r}
\def\xrr{(\sum_{i=1}^T|x_i|^r)^{\frac{1}{r}}}
\def\R{\mathbb R}
\def\ff{\frac}
\def\ss{\sqrt}
\def\B{\mathbf B}
\def\N{\mathbb N}
\def\kk{\kappa} \def\m{{\bf m}}
\def\dd{\delta} \def\DD{\Dd} \def\vv{\varepsilon} \def\rr{\rho}
\def\<{\langle} \def\>{\rangle} \def\GG{\Gamma} \def\gg{\gamma}
  \def\nn{\nabla} \def\pp{\partial} \def\EE{\scr E}
\def\d{\text{\rm{d}}} \def\bb{\beta} \def\aa{\alpha} \def\D{\scr D}
  \def\si{\sigma} \def\ess{\text{\rm{ess}}}\def\lam{\lambda}
\def\beg{\begin} \def\beq{\begin{equation}}  \def\F{\scr F}
\def\Ric{\text{\rm{Ric}}} \def \Hess{\text{\rm{Hess}}}
\def\e{\text{\rm{e}}} \def\ua{\underline a} \def\OO{\Omega}  \def\oo{\omega}
 \def\tt{\tilde} \def\Ric{\text{\rm{Ric}}}
\def\cut{\text{\rm{cut}}} \def\P{\mathbb P} \def\ifn{I_n(f^{\bigotimes n})}
\def\C{\scr C}      \def\alphaa{\mathbf{r}}     \def\r{r}
\def\gap{\text{\rm{gap}}} \def\prr{\pi_{{\bf m},\varrho}}  \def\r{\mathbf r}
\def\Z{\mathbb Z} \def\vrr{\varrho} \def\l{\lambda}
\def\L{\scr L}\def\Tilde{\tilde} \def\TILDE{\tilde}\def\II{\mathbb I}
\def\i{{\rm in}}\def\Sect{{\rm Sect}}\def\E{\mathbb E} \def\H{\mathbb H}
\def\M{\scr M}\def\Q{\mathbb Q} \def\texto{\text{o}} \def\LL{\Lambda}
\def\Rank{{\rm Rank}} \def\B{\scr B} \def\i{{\rm i}} \def\HR{Hat{\R}^d}
\def\to{\rightarrow}\def\l{\ell}\def\ll{\lambda}
\def\8{\infty}\def\ee{\epsilon} \def\Y{\mathbb{Y}} \def\lf{\lfloor}
\def\rf{\rfloor}\def\3{\triangle}\def\H{\mathbb{H}}\def\S{\mathbb{S}}\def\1{\lesssim}
\def\va{\varphi}

\def\R{\mathbb R}  \def\ff{\frac} \def\ss{\sqrt} \def\B{\mathbf
B}
\def\N{\mathbb N} \def\kk{\kappa} \def\m{{\bf m}}
\def\dd{\delta} \def\DD{\Delta} \def\vv{\varepsilon} \def\rr{\rho}
\def\<{\langle} \def\>{\rangle} \def\GG{\Gamma} \def\gg{\gamma}
  \def\nn{\nabla} \def\pp{\partial} \def\EE{\scr E}
\def\d{\text{\rm{d}}} \def\bb{\beta} \def\aa{\alpha} \def\D{\scr D}
  \def\si{\sigma} \def\ess{\text{\rm{ess}}}
\def\beg{\begin} \def\beq{\begin{equation}}  \def\F{\scr F}
\def\Ric{\text{\rm{Ric}}} \def\Hess{\text{\rm{Hess}}}
\def\e{\text{\rm{e}}} \def\ua{\underline a} \def\OO{\Omega}  \def\oo{\omega}
 \def\tt{\tilde} \def\Ric{\text{\rm{Ric}}}
\def\cut{\text{\rm{cut}}} \def\P{\mathbb P} \def\ifn{I_n(f^{\bigotimes n})}
\def\C{\scr C}      \def\aaa{\mathbf{r}}     \def\r{r}
\def\gap{\text{\rm{gap}}} \def\prr{\pi_{{\bf m},\varrho}}  \def\r{\mathbf r}
\def\Z{\mathbb Z} \def\vrr{\varrho} \def\ll{\lambda}
\def\L{\scr L}\def\Tt{\tt} \def\TT{\tt}\def\II{\mathbb I}
\def\i{{\rm in}}\def\Sect{{\rm Sect}}\def\E{\mathbb E} \def\H{\mathbb H}
\def\M{\scr M}\def\Q{\mathbb Q} \def\texto{\text{o}} \def\LL{\Lambda}
\def\Rank{{\rm Rank}} \def\B{\scr B} \def\i{{\rm i}} \def\HR{\hat{\R}^d}
\def\to{\rightarrow}\def\l{\ell}
\def\8{\infty}\def\X{\mathbb{X}}\def\3{\triangle}
\def\V{\mathbb{V}}\def\M{\mathbb{M}}\def\W{\mathbb{W}}\def\Y{\mathbb{Y}}\def\1{\lesssim}

\def\La{\Lambda}\def\S{\mathbf{S}}
\def\va{\varphi}
\def\l{\lambda}
\def\var{\varphi}
\renewcommand{\bar}{\overline}
\renewcommand{\hat}{\widehat}
\renewcommand{\tilde}{\widetilde}

\maketitle
\begin{abstract}
In this paper,  under a one-sided Lipschitz condition on the drift coefficient we adopt (via contraction principle) a exponential
approximation argument to investigate large deviations for  neutral stochastic functional differential
equations. 
\end{abstract}


\section{introduction}
As is well known, Large deviation principle (LDP for short) is a branch of probability theory that deals with
the asymptotic behaviour of rare events, and it has a wide range of applications, such as mathematic finance,
statistic mechanics, biology  and so on. So the large deviation principle for SDEs
has been investigated extensively; see, e.g.;\cite{BaoJ, BYG, MZ} and reference therein.

From the literature, we know there are two main methods to investigate the LDPs, one method is based on contraction
principle in LDPs, that is, it relies on approximation arguments and exponential-type probability estimates; see
e.g.,\cite{ BZ, FW, GPP, HMS, LP, LZ, MZ, RZ} and references therein. 
\cite{FW,  LZ, RZ} concerned about the LDP for SDEs driven by Brownian motion
or Poisson measure, \cite{HMS} investigated how rapid-switching behaviour of solution($X_t^\ep$) affects the small-noise asymptotics of $X_t^\ep$-modulated diffusion processes on the certain interval.
\cite{GPP} investigated the LDP for invariant distributions of memory gradient diffusions.

The other one is weak convergence method, which  has also been applied in establishing LDPs for a
various stochastic dynamic systems; see e.g.,\cite{BaoJ, BYG, BAC, BDP, APA, Bud}. According to the compactness argument in this method of the solution space of corresponding skeleton equation, the weak convergence is done for Borel measurable functions
whose existence is based on Yamada-Watanabe theorem. In \cite{BAC, BDP, Bud}, the authors study a large deviation principle for SDEs/SPDEs.

Compared with the weak convergence method, there are few literature about the LDP for SFDEs, \cite{MZ} gave result about large
deviations for SDEs with point delay,  and large deviations for perturbed reflected diffusion processes was investigated in \cite{BZ}.
The aim of this paper is to study the LDP for NSFDEs, which extends the result in \cite{MZ}.

The structure of this paper is as follows.
In section 2, we introduce some preliminary results and notation. In section 3, we state the main result about LDP for NSFDEs and give its proof. 

Before giving the preliminaries, a few words about the notation are in order. Throughout this paper, $C>0$ stipulates a generic constant, which might change from line to line and depend on the time parameters.

\section{Preliminaries}
Let $(\mathbb{R}^d,\langle\cdot,\cdot\rangle,|\cdot|)$ be the $d$-dimensional Euclidean space with the inner product
$\langle\cdot,\cdot\rangle$ which induces the norm $|\cdot|$. Let $\mathbb{M}^{d\times d}$ denote the set of all $d\times d$
matrices, which equipped with the Hilbert-Schimidt norm $\|\cdot\|_{HS}$. $A^*$ stands for the transpose of the matrix $A$.
For a sub-interval $\mathbb{U}\subseteq\mathbb{R}$, $C(\mathbb{U};\mathbb{R}^d)$ be the family of all continuous functions
$f:\mathbb{U}\rightarrow\mathbb{R}^d$. Let $\tau>0$ be a fixed number and $\mathscr{C}=C([-\tau,0];\mathbb{R}^d)$,  endowed
with the uniform norm $\|f\|_\8:=\sup_{-\tau\le\theta\le0}|f(\theta)|$. 
Fixed $t\ge0$, let $f_t\in\mathscr{C}$ be defined by $f_t(\theta)=f(t+\theta), \theta\in[-\tau,0]$.
 In terminology, $(f_t)_{t\ge0}$ is called the segment (or window) process corresponding to $(f(t))_{t\ge-\tau}$.

In this paper, we are interested in the following neutral stochastic functional differential equation (NSFDE)
\begin{equation}\label{eq1.1}
\d\{X^\ep(t)-G(X_t^\ep)\}=b(X_t^\ep)\d t+\ss\ep\sigma(X_t^\ep)\d W(t),~~t\in[0,T],  ~~X_0^\ep=\xi\in\mathscr{C},
\end{equation}
where $G,b:\mathscr{C}\rightarrow\R^d$, $\sigma:\mathscr{C}\rightarrow\R^d\times\R^d$ and $\{W(t)\}_{t\ge0}$ is a
$d$-dimensional Brownian motion on some filtered probability space $(\Omega,\mathscr{F},(\mathscr{F}_t))_{t\ge0},\mathbb{P})$.

The proof of main result (Theorem \ref{th1}) will be based on an extension of the contraction principle
in \cite[Theorem 4.2.23]{DZ}. To make the content self-contained, we recall it as follows:
\begin{lem}\label{lem1}
Let $\{\mu_\ep\}$ be a family of probability measures that satisfies the LDP with a good rate function $I$
on a Hausdorff topological space
$\mathcal{X}$, and for $m=1,2,\cdots,$ let $f_m:\mathcal{X}\rightarrow\mathcal{Y}$ be continuous functions, with
$(\mathcal{Y},d)$ a metric space. Assume there exists a measurable map $f:\mathcal{X}\rightarrow\mathcal{Y}$ such that
for every $\alpha<\8$,
\begin{equation}\label{eq1.2}
\limsup_{m\rightarrow\8}\sup_{\{x:I(x)\le\alpha\}}d(f_m(x),f(x))=0.
\end{equation}
Then any family of probability measures $\{\tilde{\mu}_\ep\}$ for which $\{\mu_\ep\circ f_m^{-1}\}$ are exponentially good approximations
satisfies the LDP in $\mathcal{Y}$ with the good rate function $I'(y)=\inf\{I(x):y=f(x)\}$.
\end{lem}
We now state the classical exponential inequality for stochastic integral, which is crucial
in proving the exponential approximation. For more details, please refer to  Stroock \cite[lemma 4.7]{S}.
\begin{lem}\label{lem2}
Let $\alpha:[0,\8)\times\Omega\rightarrow \mathbb{R}^d\times\mathbb{R}^d$ and $\beta:[0,\8)\times\Omega\rightarrow\mathbb{R}^d$
be $(\mathscr{F}_t)_{t\ge0}$-progressively measurable processes. Assume that $\|\alpha(\cdot)\|_{HS}\le A$ and $|\beta|\le B$.
Set $\xi(t):=\int_0^t\alpha(s)\d W(s)+\int_0^t\beta(s)\d s$ for $t\ge0$. Let $T>0$ and $R>0$ satisfy $d^{\frac{1}{2}}BT<R$.
Then
\begin{equation}\label{eq1.3}
P\Big(\sup_{0\le t\le T}|\xi(t)|\ge R\Big)\le 2d\exp{\Big(\frac{-(R-d^{\frac{1}{2}}BT)^2}{2A^2dT}\Big)}.
\end{equation}
\end{lem}

\section{LDP for NSFDE}

Let $H$ denote the Cameron-Martin space, i.e.
\begin{equation*}
H=\Big\{h(t)=\int_0^t\dot{h}(s)\d s:[0,T]\rightarrow\R^d;\int_0^T|\dot{h}(s)|^2\d s<+\8\Big\},
\end{equation*}
which is an Hilbert space endowed with the inner product as follows:
\begin{equation*}
\langle f,g\rangle_H=\int_0^T \dot{f}(s)\dot{g}(s)\d s.
\end{equation*}
We define
\begin{equation}\label{eqi}
L_T(h)=\begin{cases}\frac{1}{2}\int_0^T|\dot{h}(t)|^2\d t,&\mbox{if}~~h\in H,\\
+\8&\mbox{otherwise}.
\end{cases}
\end{equation}
The well-known Schilder theorem states that the laws $\mu_\ep$ of $\{\ss\ep W(t)\}_{t\in[0,T]}$ satisfies the LDP
on $C([0,T];\R^d)$ with the rate function $L_T(\cdot)$.

To investigate the  LDP for the laws of $\{X^\ep(t)\}_{t\in[-\tau,T]}$,
we give the following assumptions about coefficients. 
\begin{enumerate}
\item[({\bf H1})]%$b, \sigma$ are continuous and bounded on $\mathscr{C}$, and
 There exists a constant $L>0$ such that 
 \begin{equation*}
 \begin{split}
&2\langle\xi(0)-\eta(0)+G(\eta)-G(\xi),b(\xi)-b(\eta)\rangle\le L\|\xi-\eta\|_\8^2,
\end{split}
\end{equation*}
and 
\begin{equation*}
\|\sigma(\xi)-\sigma(\eta)\|_{HS}^2\le L\|\xi-\eta\|_\8^2,
 ~~\xi,\eta\in\mathscr{C};
\end{equation*}
\item[(\bf H2)] There exists a constant $\kappa\in(0,1)$ such that
\begin{equation}\label{eqg}
\begin{split}
&|G(\xi)-G(\eta)|\le\kappa\|\xi-\eta\|_\8,\\
&G(0)=0,~~~~~~\xi,\eta\in\mathscr{C}.
\end{split}
\end{equation}
\end{enumerate}

\begin{rem}
The one-sided Lipschitz condition on the drift coefficient  in ({\bf H1}) is different from the global Lipschitz condition 
in  \cite{BaoJ}. Moreover, our method below is different from that of \cite{BaoJ}.
\end{rem}

\begin{rem}\label{r1}
From ({\bf H1}), ({\bf H3}), it is easy to see that
\begin{equation}\label{eqr1}
2\langle\xi(0)-G(\xi),b(\xi)\rangle\le L_2(1+\|\xi\|_\8^2),~~|G(\xi)|^2\le\kappa^2\|\xi\|^2,~\xi\in\mathscr{C}.
\end{equation}
\end{rem}
\begin{rem}\label{reme}
Let $\mu(\d\theta)\in\mathscr{P}([-\tau,0])$ and let 
\begin{equation*}
\begin{split}
&G(\xi)=\alpha_1\int_{-\tau}^0\xi(\theta)\mu(\d\theta),
~~~\sigma(\xi)=\alpha_2\int_{-\tau}^0\xi(\theta)\mu(\d\theta),\\
&b(\xi)=-\alpha_3\xi(0)-\alpha_4\Big(
\xi(0)-\alpha_1\int_{-\tau}^0\xi(\theta)\mu(\d\theta)
\Big)^{1/3}+\alpha_5\int_{-\tau}^0\xi(\theta)\mu(\d\theta),
\end{split}
\end{equation*}
for some constants $\alpha_i,i=1,\cdots,5$ such that
$\alpha_1\le\kappa$, $\Big(\alpha_3(\alpha_1-1)+\alpha_5(1+\alpha_1)\Big)\vee \alpha_2^2\le L$, then the assumptions ({\bf H1}) and  ({\bf H2}) hold true.
In fact, 
by the H\"older inequality, one has
\begin{equation*}
\begin{split}
&|G(\xi)-G(\eta)|^2\le\alpha_1^2\int_{-\tau}^0|\xi(\theta)-\eta(\theta)|^2\mu(\d\theta)\le\alpha_1^2\|\xi-\eta\|_\8^2\int_{-\tau}^0\mu(\d\theta)=\alpha_1^2\|\xi-\eta\|_\8^2,\\
&\mbox{noting that}\\
&-\alpha_4\langle\xi(0)-\eta(0)-(G(\xi)-G(\eta)),(\xi(0)-G(\xi))^{1/3}-(\eta(0)-G(\eta))^{1/3}\rangle\le 0,\\
&\mbox{so}\\
&\langle\xi(0)-\eta(0)-(G(\xi)-G(\eta)),b(\xi)-b(\eta)\rangle\\
&\le-\alpha_3|\xi(0)-\eta(0)|^2+\alpha_3|\xi(0)-\eta(0)||G(\xi)-G(\eta)|\\
&~~+\alpha_5|\xi(0)-\eta(0)|\int_{-\tau}^0|\xi(\theta)-\eta(\theta)|\mu(\d\theta)-\alpha_5|G(\xi)-G(\eta)|\int_{-\tau}^0|\xi(\theta)-\eta(\theta)|\mu(\d\theta)\\
&\le \alpha_3(\alpha_1-1)+\alpha_5(1+\alpha_1)\|\xi-\eta\|_\8^2,\\
&\|\sigma(\xi)-\sigma(\eta)\|_{HS}^2\le\alpha_2^2\int_{-\tau}^0|\xi(\theta)-\eta(\theta)|^2\mu(\d\theta)\le \alpha_2^2\|\xi-\eta\|_\8^2.
\end{split}
\end{equation*}
Therefore, the assumptions hold if the constants $\alpha_i, i=1, \ldots, 5$ satisfy the conditions above.
\end{rem}

Let $F(h)$ be the unique solution of the following deterministic equation:
\begin{equation}\label{eq1.4}
\begin{cases}
&F(h)(t)-G(F_t(h))=F(h)(0)-G(F_0(h))+\int_0^t b\Big(F_s(h)\Big)\d s\\
&~~~~~~~~~~~~~~~~~~~~~~~~~~~~+\int_0^t\sigma\Big(F_s(h)\Big)\dot{h}(s)\d s,~~~t\in[0,T],\\
&F_0(h)=\xi(\theta),~~\theta\in[-\tau,0].
\end{cases}
\end{equation}
Herein, $F_t(h)(\theta)=F(h)(t+\theta)$, $\theta\in[-\tau,0]$.

The main result of this section is stated as follows.
\begin{thm}\label{th1}
%Let $\mu_\ep$ be the law of $X^\ep(\cdot)$ on $C([-\tau,T];\mathbb{R}^d)$, equipped with the uniform topology.
Under the assumptions ({\bf H1})-({\bf H2}),
 it holds that $\{\mu_\ep,\ep>0\}$, the law of $X^\ep(\cdot)$ on $C([-\tau,T];\R^d)$,  satisfies the large deviation principle with the rate function below
\begin{equation}\label{eq1.5}
I(f):=\inf\Big\{L_T(h);F(h)=f,h\in H\Big\},~~~f\in C([-\tau,T];\mathbb{R}^d),
\end{equation}
where $L_T(h)$ is defined as in \eqref{eqi}.
That is,
\begin{enumerate}
\item[(i)] for any closed subset $C\subset C([-\tau, T];\mathbb{R}^d)$,
\begin{equation*}
\limsup_{\ep\rightarrow0}\log\mu_\ep(C)\le-\inf_{f\in C}I(f),
\end{equation*}
\item[(ii)] for any open subset $G\subset C([-\tau, T];\mathbb{R}^d)$,
\begin{equation*}
\liminf_{\ep\rightarrow0}\log\mu_\ep(G)\ge-\inf_{f\in G}I(f).
\end{equation*}
\end{enumerate}
\end{thm}
Before giving the proof of Theorem \ref{th1}, we prepare some lemmas.

%Let $X^\ep(\cdot)$ be the solution of \eqref{eq1.1} and
 We construct $X^{\ep,n}(\cdot)$ by exploiting an approximate  scheme, that is, for
 a real positive number $s$, let $[s]=\sup\{k\in \Z:k\le s\}$ be its integer part. For any $n\in N_0$,
  we consider the following NSFDE
\begin{equation}\label{eq1.6}
\d \{X^{\ep,n}(t)-G(X_t^{\ep,n})\}=b(X_t^{\ep, n})\d t+\ss\ep\sigma(\hat{X}_t^{\ep,n})\d W(t),~~t\ge0,~~X_0^{\ep,n}=\xi,
\end{equation}
where, for $t\ge 0$,
\begin{equation*}
\hat{X}_t^{\ep,n}(\theta):=X^{\ep,n}((t+\theta)\wedge t_n), ~~~ t_n:=[nt]/n, n\ge1, \theta\in[-\tau,0].
\end{equation*}
According to \cite[Theorem 2.2, p.204]{M}, \eqref{eq1.6} has a unique solution by solving piece-wisely
with the time length $1/n$.

In the sequel, we consider two cases separately.

{\bf Case 1: }.
We assume that $b,\sigma$ are bounded, i.e.
\begin{enumerate}
\item[({\bf H3})] There exists a constant $M>0$ such that
\begin{equation*}
|b(\xi)|\vee\|\sigma(\xi)\|_{HS}\le M, \forall \xi\in\mathscr{C}.
\end{equation*}
\end{enumerate}
Next, we show that $\{X^{\ep,n},\ep>0\}$ defined by \eqref{eq1.6} approximates to
%are exponentially good approximations of
$\{X^\ep,\ep>0\}$.
\begin{lem}\label{lem3}
Assume ({\bf H1}), ({\bf H2}), and ({\bf H3}) hold,  then for any $\delta>0$, one has
\begin{equation}\label{eqq1}
\lim_{n\rightarrow\8}\limsup_{\ep\rightarrow 0}\ep\log P\Big(\sup_{-\tau\le t\le T}|X^\ep(t)-X^{\ep,n}(t)|>\delta\Big)=-\8.
\end{equation}
\end{lem}
\begin{proof}
For notation brevity, we set $Z^{\ep,n}(t):=X^\ep(t)-X^{\ep,n}(t), t\ge0$ and $Y^{\ep,n}(t):=X^\ep(t)-X^{\ep,n}(t)-(G(X_t^\ep)-G(X_t^{\ep,n})), t\ge0$.
Noting $X^{\ep,n}_0=X^{n}_0=\xi,$ we write $Y^{\ep,n}(t)$ as follows:
\begin{equation*}
Y^{\ep,n}(t)=\int_0^t(b(X_s^\ep)-b(X_s^{\ep,n}))\d s
+\ss\ep\int_0^t(\sigma(X_s^\ep)-\sigma(\hat{X}_s^{\ep,n}))\d W(s).
\end{equation*}
It is easy to see from \eqref{eqg} that
\begin{equation*}
\begin{split}
|Z^{\ep,n}(t)|&\le|Y^{\ep,n}(t)|+|G(X_t^\ep)-G(X_t^{\ep,n})|\\
&\le|Y^{\ep,n}(t)|+\kappa\|X_t^\ep-X_t^{\ep,n}\|_\8,
\end{split}
\end{equation*}
and noting $X_0^\ep=X_0^{\ep,n}=\xi,  \xi\in\mathscr{C}$, it yields that
\begin{equation}\label{eqa}
\sup_{0\le t\le T}|Z^{\ep,n}(t)|\le \frac{1}{1-\kappa}\sup_{0\le t\le T}|Y^{\ep,n}(t)|.
\end{equation}
For $\rho>0$, we define
$\tau_{n_\rho}^\ep
=\inf\{t\ge 0: \|X_t^{\ep,n}-\hat{X}_t^{\ep,n}\|_\8>\rho\}
$,
$Z^{\ep,n_{\rho}}=Z^{\ep,n}(t\wedge\tau_{n_\rho}^\ep)$,
   $\xi_{n_\rho}^\ep=\inf\{t\ge0: |Z^{\ep,n_\rho}(t)|\ge\delta\}$, and compute
 \begin{equation}\label{eq1.7}
 \begin{split}
 P\Big(\sup_{0\le t\le T}|Z^{\ep,n}(t)|>\delta\Big)&= P\Big(\sup_{0\le t\le T}|Z^{\ep,n}(t)|>\delta,\tau_{n_\rho}^\ep\le T\Big)
 +P\Big(\sup_{0\le t\le T}|Z^{\ep,n}(t)|>\delta,\tau_{n_\rho}^\ep> T\Big)\\
 &\le P(\tau_{n_\rho}^\ep\le T)+P\Big(\sup_{0\le t\le T}|Z^{\ep,n}(t)|>\delta,\tau_{n_\rho}^\ep> T\Big)\\
 &\le P(\tau_{n_\rho}^\ep\le T)+P(\xi_{n_\rho}^\ep\le T).
 \end{split}
 \end{equation}
 Observe that
 \begin{equation*}
 \begin{split}
 X_t^{\ep,n}(\theta)-\hat{X}_t^{\ep,n}(\theta)&=X^{\ep,n}(t+\theta)-X^{\ep,n}((t+\theta)\wedge t_n)\\
 &=(X^{\ep,n}(t+\theta)-X^{\ep,n}(t+\theta))I_{\{(t+\theta)<t_n\}}
 +(X^{\ep,n}(t+\theta)-X^{\ep,n}(t_n))I_{\{t_n\le(t+\theta)\}}\\
 &=(X^{\ep,n}(t+\theta)-X^{\ep,n}(t_n))I_{\{t_n\le(t+\theta)\}}\\
 &=G(X_{t+\theta}^{\ep,n})-G(X_{t_n}^{\ep,n})
 +\Big(\int_{t_n}^{t+\theta}b(X_s^{\ep,n})\d s+\int_{t_n}^{t+\theta}\ss\ep\sigma(\hat{X}_s^{\ep,n})\d W(s)\Big).
 \end{split}
 \end{equation*}
 This, together with \eqref{eqg}, yields
 \begin{equation}\label{eq1.8}
 \sup_{0\le t\le T}\|X_t^{\ep,n}-\hat{X}_t^{\ep,n}\|_\8\le\frac{1}{1-\kappa}\sup_{0\le t\le T}\sup_{t_n-t\le\theta\le 0}\Big|\int_{t_n}^{t+\theta}b(X_s^{\ep,n})\d s+\int_{t_n}^{t+\theta}\ss\ep\sigma(\hat{X}_s^{\ep,n})\d W(s)\Big|.
 \end{equation}
  Taking ({\bf H3}) into consideration and utilizing Lemma \ref{lem2}, one gets that
\begin{equation*}
P\Big(\sup_{0\le t\le T}\|X_t^{\ep,n}-\hat{X}_t^{\ep,n}\|_\8\ge\rho\Big)
\le2d\exp\Big(-\frac{(n\rho(1-\kappa)-\ss\d M)^2}{2nM^2(1-\kappa)^2d\ep}\Big),
\end{equation*}
provided that $\frac{\ss dM}{(1-\kappa)n}<\rho$.
Which, together with the definition of stopping time $\tau_{n_\rho}^\ep$, it follows that
\begin{equation}\label{eq1.11}
\lim_{n\rightarrow\8}\limsup_{\ep\rightarrow 0}\ep\log P(\tau_{n_\rho}^\ep\le T)=-\8.
\end{equation}
For $\lambda>0$, let $\phi_\lambda(y)=(\rho^2+|y|^2)^\lambda$, an application of It\^o's formula yields 
\begin{equation}\label{equ}
\phi_\lambda(Y^{\ep,n_\rho}(t))=\rho^{2\lambda}+M^{\ep,n_\rho}(t)+\int_0^{t\wedge\tau_{n_\rho}^\ep}\gamma_\lambda^\ep(s)\d s,
\end{equation}
where $M^{\ep,n_\rho}(t): =2\lambda\int_0^{t\wedge\tau_{n_\rho}^\ep}(\rho^2+|Y^{\ep,n}(s)|^2)^{\lambda-1}
\ss\ep\langle Y^{\ep,n}(s),\sigma(X_s^{\ep,n})-\sigma(\hat{X}_s^{\ep,n})\d W(s)\rangle$ is a martingale. Moreover, 
by ({\bf H1}), we see that
\begin{equation}\label{eq1.9}
\begin{split}
\gamma_\lambda^\ep(s):&=2\lambda(\rho^2+|Y^{\ep,n}(s)|^2)^{\lambda-1}\langle Y^{\ep,n}(s),b(X_s^\ep)-b(X_s^{\ep,n})\rangle\\
&~~+2\lambda(\lambda-1)\ep(\rho^2+|Y^{\ep,n}(s)|^2)^{\lambda-2} |(\sigma(X_s^\ep)-\sigma(\hat{X}_s^{\ep,n}))^*Y^{\ep,n}(s)|^2\\
&~~+\lambda\ep (\rho^2+|Y^{\ep,n}(s)|^2)^{\lambda-1}\|\sigma(X_s^\ep)-\sigma(\hat{X}_s^{\ep,n})\|_{HS}^2\\
&\le2L\lambda(\rho^2+|Y^{\ep,n}(s)|^2)^{\lambda-1}\|Z_s^{\ep,n}\|_\8^2+\lambda(2\lambda-1)\ep(\rho^2+|Y^{\ep,n}(s)|^2)^{\lambda-1} \|(\sigma(X_s^\ep)-\sigma(\hat{X}_s^{\ep,n}))\|_{HS}^2\\
&\le C_1(\rho^2+|Y^{\ep,n}(s)|^2)^{\lambda-1}\|Z_s^{\ep,n}\|_\8^2+C_2(\rho^2+|Y^{\ep,n}(s)|^2)^{\lambda-1} \|X_s^{\ep,n}-\hat{X}_s^{\ep,n}\|_\8^2,
\end{split}
\end{equation}
where $C_1=2L\lambda[(2\lambda-1)\ep+1],~~C_2=2L\lambda\ep(2\lambda-1)$.

Using the Burkholder-Davis-Gundy (BDG for short) inequality, we obtain
\begin{equation}\label{eq1.10}
\begin{split}
&\E\Big(\sup_{0\le t\le T}M^{\ep,n_\rho}(t)\Big)\\
&\le8\ss{2\ep}\lambda\Big(\E\int_0^{T\wedge\tau_{n_\rho}^\ep}(\rho^2+|Y^{\ep,n}(s)|^2)^{2\lambda-2}
|Y^{\ep,n}(s)|^2\|\sigma(X_s^\ep)-\sigma(\hat{X}_s^{\ep,n})\|_{HS}^2\d s\Big)^{\frac{1}{2}}\\
&\le\frac{1}{2}\E\Big(\sup_{0\le t\le T\wedge\tau_{n_\rho}^\ep}(\rho^2+|Y^{\ep,n}(s)|^2)^\lambda\Big)
+64\lambda^2\ep\E\int_0^{T\wedge\tau_{n_\rho}^\ep}(\rho^2+|Y^{\ep,n}(s)|^2)^{\lambda-1}
\|\sigma(X_s^\ep)-\sigma(\hat{X}_s^{\ep,n})\|_{HS}^2\d s\\
&\le\frac{1}{2}\E\Big(\sup_{0\le t\le T\wedge\tau_{n_\rho}^\ep}(\rho^2+|Y^{\ep,n}(s)|^2)^\lambda\Big)
+128L\lambda^2\ep\E\int_0^{T\wedge\tau_{n_\rho}^\ep}(\rho^2+|Y^{\ep,n}(s)|^2)^{\lambda-1}\|Z_s^{\ep,n}\|_\8^2\d s\\
&~~~~+128\lambda^2\ep\E\int_0^{T\wedge\tau_{n_\rho}^\ep}(\rho^2+|Y^{\ep,n}(s)|^2)^{\lambda-1}
\|\sigma(X_s^{\ep,n})-\sigma(\hat{X}_s^{\ep,n})\|_{HS}^2\d s\\
&\le\frac{1}{2}\E\Big(\sup_{0\le t\le T\wedge\tau_{n_\rho}^\ep}(\rho^2+|Y^{\ep,n}(s)|^2)^\lambda\Big)
+128L\lambda^2\ep\E\int_0^{T\wedge\tau_{n_\rho}^\ep}(\rho^2+|Y^{\ep,n}(s)|^2)^{\lambda-1}\|Z_s^{\ep,n}\|_\8^2\d s\\
&~~~~+128L\lambda^2\ep\E\int_0^{T\wedge\tau_{n_\rho}^\ep}(\rho^2+|Y^{\ep,n}(s)|^2)^{\lambda-1}
\|X_s^{\ep,n}-\hat{X}_s^{\ep,n}\|_\8^2\d s.
\end{split}
\end{equation}

Combining \eqref{eq1.9} and \eqref{eq1.10} and reformulating \eqref{equ},  one has
\begin{equation}
\begin{split}
&\E\Big(\sup_{0\le t\le T}\phi_\lambda(Y^{\ep,n_\rho}(t))\Big)\\
%&\le\rho^{2\lambda}+\E\Big(\sup_{0\le t\le T}M^{\ep,n_\rho}(t)\Big)
%+\E\Big(\sup_{0\le t\le T}\int_0^{t\wedge\tau_{n_\rho}^\ep}\gamma_\lambda^\ep(s)\d s\Big)\\
&\le2\rho^{2\lambda}+4L\lambda(66\lambda\ep-\ep+1)
\int_0^T\E(\rho^2+|Y^{\ep,n_\rho}(s)|^2)^{\lambda-1}\|Z_s^{\ep,n_\rho}\|_\8^2\d s\\
&~~~~+4L\lambda\ep(68\lambda-1)\int_0^T\E(\rho^2+|Y^{\ep,n_\rho}(s)|^2)^{\lambda-1}
\|X_s^{\ep,n_\rho}-\hat{X}_s^{\ep,n_\rho}\|_\8^2\d s\\
&\le2\rho^{2\lambda}
+4L\lambda(66\lambda\ep-\ep+1)\int_0^T\E\Big(\sup_{0\le u\le s}(\rho^2+|Y^{\ep,n_\rho}(u)|^2)^{\lambda-1}\|Z_u^{\ep,n_\rho}\|_\8^2\Big)\d s\\
&~~~~+4L\lambda\ep(68\lambda-1)\int_0^T\E\Big(\sup_{0\le u\le s}(\rho^2+|Y^{\ep,n_\rho}(u)|^2)^{\lambda-1}
\|X_u^{\ep,n_\rho}-\hat{X}_u^{\ep,n_\rho}\|_\8^2\Big)\d s\\
&\le2\rho^{2\lambda}+(C_3+C_4)\int_0^T\E\Big(\sup_{0\le u\le s}(\rho^2+|Y^{\ep,n_\rho}(u)|^2)^{\lambda}\Big)\d s,
\end{split}
\end{equation}
where $C_3=4\lambda(66\lambda\ep-\ep+1)
\frac{L}{(1-\kappa)^2}$,~~$C_4=4L\lambda\ep(68\lambda-1)$.
In the last step, we utilized the fact that $Y^{\ep,n_\rho}(t)=0, t\in[-\tau,0]$ and \eqref{eqa}.

Choosing $\lambda=\frac{1}{\ep}$ and setting $\Phi^{\ep,n_\rho}(t):=(\rho^2+|Y^{\ep,n_\rho}(t\wedge\xi_{n_\rho}^\ep)|^2)^{1/\ep}$,
 by the Gronwall inequality, we obtain
 \begin{equation*}
 \begin{split}
 \E\Big(\sup_{0\le t\le T}\Phi^{\ep,n_\rho}(t)\Big)
 \le2\rho^{2\lambda}\e^{(C_3+C_4)T}
 \le2\rho^{2/\ep}\e^{C_5T/\ep},
 \end{split}
 \end{equation*}
 where $C_5=L\Big(\frac{268}{(1-\kappa)^2}+272\Big)$.
Noting that
\begin{equation*}
\Phi^{\ep,n_\rho}(t)=(\rho^2+|Y^{\ep,n_\rho}(t)|^2)^{1/\ep}
I_{\{t\le\xi_{n_\rho}^\ep\}}
+(\rho^2+|Y^{\ep,n_\rho}(\xi_{n_\rho}^\ep)|^2)^{1/\ep}I_{\{\xi_{n_\rho^\ep<t}\}},
\end{equation*}
so
\begin{equation*}
(\rho^2+(1-\kappa)^2\delta^2)^{1/\ep}P(\xi_{n_\rho}^\ep\le T)\le \E\Big(\sup_{0\le t\le T}\Phi^{\ep,n_\rho}(t)\Big),
\end{equation*}
then we have
\begin{equation*}
P(\xi_{n_\rho}^\ep\le T)\le\Big(\frac{2^\ep\rho^2}{\rho^2+(1-\kappa)^2\delta^2}\Big)^{1/\ep}\e^{C_5T/\ep}.
\end{equation*}
Thus,
\begin{equation*}
\limsup_{\ep\rightarrow0}\ep\log P(\xi_{n_\rho}^\ep\le T)\le\log\Big(\frac{\rho^2}{\rho^2+(1-\kappa)^2\delta^2}\Big)+C_5T.
\end{equation*}
Finally, given $L>0$, choose $\rho$ sufficiently small such that $\log\Big(\frac{\rho^2}{\rho^2+(1-\kappa)^2\delta^2}\Big)+C_5T\le -2L$.
Next, utilizing \eqref{eq1.11}, choose $N$ such that $\limsup_{\ep\rightarrow 0}\ep\log P(\tau_{n_\rho}^\ep\le T)\le -2L$ for
$n\ge N$. Then, for $n\ge N$ there is an $0<\ep_n<1$ such that $P(\tau_{n_\rho}^\ep\le T)\le\e^{-L/\ep}$ and
$P(\xi_{n_\rho}^\ep\le T)\le\e^{-L/\ep}$ for $0<\ep\le\ep_n$, so \eqref{eq1.7} leads to
\begin{equation*}
P\Big(\sup_{0\le t\le T}|Z^{\ep,n}(t)|\ge\delta\Big)\le 2\e^{-L/\ep},~~0<\ep\le\ep_n.
\end{equation*}
Thus,
\begin{equation*}
\limsup_{\ep\rightarrow 0}\ep\log P\Big(\sup_{0\le t\le T}|Z^{\ep,n}(t)|>\delta\Big)\le -L,~~~~n\ge N.
\end{equation*}
The proof of the lemma is complete.
\end{proof}
For $n\ge1$, define the map $F^n(\cdot):C_0([0,T],\mathbb{R}^d)\rightarrow C_\xi([-\tau,T],\mathbb{R}^d)$ by
\begin{equation*}
\begin{cases}
&F^n(\omega)(t)-G(F_t^n(\omega))=F^n(\omega)(t_n)-G(F_{t_n}(\omega))+\int_{t_n}^tb(F_s^n(\omega))\d s\\
&~~~~~~~~~~~~~~~~+\sigma(\hat{F}_s^n(\omega))(\omega(t)-\omega(t_n)),~~t_n\le t\le t_n+\frac{1}{n},\\
&F^n(\omega)(t)=\xi(t),~~-\tau\le t\le 0,
\end{cases}
\end{equation*}
where $F_s^n(\omega)(\theta)=F^n(\omega)(s+\theta)$ and $\hat{F}_s^n(\omega)(\theta)=\hat{F}^n(\omega)((s+\theta)\wedge s_n)$.

Notice that, $X^{\ep,n}(t)=F^n(\ss\ep W)(t)$, which is a continuous map. Herein, $W$ is a standard Brownian motion.
For $h\in H$, we define
\begin{equation}\label{eq1.12}
\begin{cases}
&F^n(h)(t)-G(F_t^n(h))=F^n(h)(0)-G(F_0^n(h))+\int_0^t b\Big(F_s^n(h)\Big)\d s\\
&~~~~~~~~~~~~~~~~~~~~~~~~~~~~~~~~
+\int_0^t\sigma\Big(\hat{F}_s^n(h)\Big)\dot{h}(s)\d s, t\in[0,T],\\
&F_0^n(h)=\xi\in\mathscr{C} .
\end{cases}
\end{equation}
The next lemma shows that the measurable map $F(h)(\cdot)$ can be approximated well by the continuous maps
$F^n(h)(\cdot)$.
\begin{lem}\label{lem4}
Under the assumptions of Theorem \ref{th1}, we have
\begin{equation}\label{eq1.15}
\lim_{n\rightarrow\8}\sup_{\{h: L_T(h)\le\alpha\}}\sup_{-\tau\le t\le T}\Big|F^n(h)(t)-F(h)(t)\Big|=0,
\end{equation}
where $\alpha<\infty$ is a constant.
\end{lem}
\begin{proof}
For notation brevity, we set $M^n(t):=F^n(h)(t)-G(F_t^n(h))$,  by  fundamental inequality $(a+b)^2\le[1+\eta](a^2+\frac{b^2}{\eta})$ and ({\bf H2}), we derive
\begin{equation*}
\begin{split}
|F^n(h)(t)|^2&=|F^n(h)(t)-G(F_t^n(h))+G(F_t^n(h))|^2\\
&\le(1+\eta)\Big(\frac{|G(F_t^n(h))|^2}{\eta}+|F^n(h)(t)-G(F_t^n(h))|^p\Big)\\
&\le(1+\eta)\Big(\frac{\kappa^2\|F_t^n(h)\|_\8^2}{\eta}+|F^n(h)(t)-G(F_t^n(h))|^2\Big).
%&\le\kappa\|F_t^n(h)\|_\8^2+\frac{1}{1-\kappa}|F^n(h)(t)-G(F_t^n(h))|^2,
\end{split}
\end{equation*}
Letting $\eta=\frac{\kappa}{1-\kappa}$,  we then have
\begin{equation}\label{eq1}
\begin{split}
\sup_{0\le t\le T}|F^n(h)(t)|^2&\le \frac{\kappa}{1-\kappa}\|\xi\|_\8^2+\frac{1}{(1-\kappa)^2}\sup_{0\le t\le T}|M^n(t)|^2.
\end{split}
\end{equation} 

 On the other hand, it is easy to see that
\begin{equation}\label{eq2}
|M^n(t)|^2\le(1+\kappa)^2\|F_t^n(h)\|_\8^2.
\end{equation}

 By ({\bf H1}), ({\bf H2}), we obtain from \eqref{eq1.12} that 
\begin{equation*}
\begin{split}
&|M^n(t)|^2\le (1+\kappa)^2\|\xi\|_\8^2+\int_0^t
2\langle M^n(s),b(F_s^n(h))+\sigma(\hat{F}_t^n(h))\dot{h}(s)\rangle\d s\\
&\le(1+\kappa)^2\|\xi\|_\8^2+L_2\int_0^t(1+\|F_s^n(h)\|_\8^2)\d s+\int_0^t|M^n(s)|^2\d s+\int_0^t|\sigma(\hat{F}_t^n(h))\dot{h}(s)|^2\d s\\
&\le(1+\kappa)^2\|\xi\|_\8^2+L_2\int_0^t(1+\|F_s^n(h)\|_\8^2)\d s+\int_0^t|M^n(s)|^2\d s\\
&+L_2\int_0^t(1+\|\hat{F}_t^n(h)\|_\8^2)|\dot{h}(s)|^2\d s.
\end{split}
\end{equation*}
Noting that $\|\hat{F}_t^n(h)\|_\8=\sup_{-\tau\le \theta\le 0}{F}^n(h)((t+\theta)\wedge t_n) \le\sup_{-\tau\le \theta\le 0}{F}^n(h)(t+\theta)$,
which together with \eqref{eq1},\eqref{eq2}, yields that
\begin{equation*}
\begin{split}
&\sup_{-\tau\le t\le T}|F^n(h)(t)|^2\\
&\le\|\xi\|_\8^2+\sup_{0\le t\le T}|F^n(h)(t)|^2\\
&\le\|\xi\|_\8^2+\frac{\kappa}{1-\kappa}\|\xi\|_\8^2+
\frac{1}{(1-\kappa)^2}\sup_{0\le t\le T}|M^n(t)|^2\\
&\le \frac{1-\kappa+(1+\kappa)^2}{(1-\kappa)^2}\|\xi\|_\8^2
+\frac{1}{(1-\kappa)^2}\Big[
(L_2+(1+\kappa)^2)\int_0^T\|F_s^n(h)\|_\8^2\d s\\
&~~~~+L_2\int_0^T\|F_s^n(h)\|_\8^2|\dot{h}(s)|^2\d s+L_2\int_0^T|\dot{h}(s)|^2\d s
\Big],
\end{split}
\end{equation*}
by the Gronwall inequality, we get
\begin{equation*}
\begin{split}
\sup_{n\ge 1}\sup_{-\tau\le t\le T}\Big|F^n(h)(t)\Big|^2&\le 
\Big(\frac{1-\kappa+(1+\kappa)^2}{(1-\kappa)^2}\|\xi\|_\8^2
+\frac{2L_2L_T(h)}{(1-\kappa)^2}\Big)
\exp\Big\{\frac{(L_2+(1+\kappa)^2)T+2L_2L_T(h)}{(1-\kappa)^2}\Big\}\\
&\le C_1(1+L_T(h))\exp\{C_2(1+L_T(h))\},
\end{split}
\end{equation*}
where $C_1=\Big(\frac{1-\kappa+(1+\kappa)^2}{(1-\kappa)^2}\|\xi\|_\8^2\Big)\vee\Big(\frac{2L_2}{(1-\kappa)^2}\Big)$, $C_2=\Big(\frac{(L_2+(1+\kappa)^2)T}{(1-\kappa)^2}\Big)\vee\Big(\frac{2L_2}{(1-\kappa)^2}\Big)$.

In particular,
\begin{equation}\label{eq1.13}
M_\alpha=\sup_{h;L_T(h)\le\alpha}\sup_{n\ge 1}\sup_{-\tau\le t\le T}
\Big|F^n(h)(t)\Big|^2\le C_1(1+\alpha)\exp\{C_2(1+\alpha)\}<\8.
\end{equation}
Hence, in the same way as the   argument of \eqref{eq1.8}, we arrive at
\begin{equation}\label{eq1.14}
\begin{split}
 \sup_{0\le t\le T}\|F_t^{n}(h)-\hat{F}_t^{n}(h)\|_\8&\le\frac{1}{1-\kappa}\sup_{0\le t\le T}\sup_{t_n-t\le\theta\le 0}\Big|\int_{t_n}^{t+\theta}b(F_s^{n}(h))\d s+\int_{t_n}^{t+\theta}\sigma(\hat{F}_s^{n}(h))\dot{h}(s)\d s\Big|\\
&\le \frac{1}{1-\kappa}\sup_{0\le t\le T}\Big(\int_{t_n}^t|b(F_s^n(h))|\d s+\int_{t_n}^t|\sigma(\hat{F}_s^n(h))\dot{h}(s)|\d s\Big)\\
&\le C_\alpha M_\alpha\Big(\frac{1}{n}\Big)^{1/2}\rightarrow 0, ~~\mbox{as}~n\rightarrow\8
\end{split}
\end{equation}
uniformly over the set $\{h;L_T(h)\le\alpha\}$.

For notation brevity, we set $D^n(h)(t):=F^n(h)(t)-F(h)(t)-(G(F_t^n(h))-G(F_t(h)))$, similarly, it is easy to see from ({\bf H1}),({\bf H2}) that
\begin{equation}\label{eq3}
\sup_{0\le t\le T}|F^n(h)(t)-F(h)(t)|^2\le\frac{1}{(1-\kappa)^2}\sup_{0\le t\le T}|D^n(h)(t)|^2,
\end{equation}
and
\begin{equation}\label{eq4}
|D^n(h)(t)|^2\le(1+\kappa)^2\|F_t^n(h)-F_t(h)\|_\8^2.
\end{equation}
Using  \eqref{eq1.4} and \eqref{eq1.12}, we deduce
\begin{equation}
\begin{split}
|D^n(h)(t)|^2
&\le\int_0^t2|\langle D^n(h)(s),b(F_s^n(h))-b(F_s(h))\rangle|\d s\\
&~~+\int_0^t2|\langle D^n(h)(s),[\sigma(\hat{F}_s^n(h)-\sigma(F_s^n(h))+\sigma(F_s^n(h))-\sigma(F_s(h))]\dot{h}(s)\rangle|\d s\\
&\le L\int_0^t\| F_s^n(h)-F_s(h)\|_\8^2\d s+\int_0^t|D^n(h)(s)|^2\d s\\
&~~+L\int_0^t |F^n(h)(s)-F(h)(s)|^2|\dot{h}(s)|^2\d s
+L\int_0^t\|\hat{F}_s^n(h))-F_s^n(h))\|_\8^2|\dot{h}(s)|^2\d s,
\end{split}
\end{equation}
which, together with \eqref{eq1.14}, \eqref{eq3} and \eqref{eq4}, yields that
\begin{equation*}
\begin{split}
\sup_{-\tau\le t\le T}|F^n(h)(t)-F(h)(t)|^2&\le
\frac{1}{(1-\kappa)^2} \Big\{(L+(1+\kappa)^2)\int_0^T\| F_s^n(h)-F_s(h)\|_\8^2\d s\\
&~~+L\int_0^T\| F_s^n(h)-F_s(h)\|_\8^2|\dot{h}(s)|^2\d s+2L\alpha C_\alpha M_\alpha\Big(\frac{1}{n}\Big)^{1/4}\Big\},
\end{split}
\end{equation*}
it follows from the Gronwall inequality that,
\begin{equation*}
\sup_{-\tau\le t\le T}|F^n(h)(t)-F(h)(s)|^2\le \frac{2L\alpha C_\alpha M_\alpha\Big(\frac{1}{n}\Big)^{1/4}}{(1-\kappa)^2}\exp\Big\{\frac{(L+(1+\kappa)^2)T+2L\alpha}{(1-\kappa)^2}\Big\}.
\end{equation*}
Hence, the desired assertion is followed by taking $n\rightarrow\8$.
\end{proof}
{\bf Proof of Theorem \ref{th1} in case 1}
\begin{proof}
Notice that $X^{\ep,n}(s)=F^n(\ep^{1/2}W)(s)$, where $W$ is the Brownian motion.
Then by the contraction principle in large deviations theory, we get that the law
of $X^{\ep,n}(s)$ satisfies an LDP. Then Lemma \ref{lem3} states that $X^{\ep,n}(s)$ 
approximates exponentially to $X^\ep(s)$. Furthermore, Lemma \ref{lem4} shows that the extension of contraction principle to measurable maps
$F(h)(\cdot)$ can be approximated well by continuous maps $F^n(h)(\cdot)$, 
i.e. Lemma \ref{lem3}, so the proof of case 1 of Theorem \ref{th1} follows from Lemma \ref{lem1}.
\end{proof}
Next, we consider

{\bf Case 2: $b,\sigma$
are unbounded.}
\begin{lem}\label{lemr}
Under ({\bf H1}), ({\bf H2}), and for $R>0$, we have
\begin{equation}\label{eq1.16}
\lim_{R\rightarrow\8}\limsup_{\ep\rightarrow 0}\ep\log P\Big(\sup_{-\tau\le t\le T}|X^\ep(t)|>R\Big)=-\8.
\end{equation}
\end{lem}
\begin{proof}
For notation brevity, we set $Y^\ep(t):=X^\ep(t)-G(X^\ep(t))$, from ({\bf H2}) and fundamental inequality, it yields that 
\begin{equation}\label{eqs}
|Y^\ep(t)|^2\le(1+\kappa)^2\|X_t^\ep\|_\8^2,
\end{equation}
and 
\begin{equation}\label{eqv}
\begin{split}
%&\sup_{0\le t\le T}|X^\ep(t)|^2\le\frac{\kappa}{1-\kappa}\|\xi\|_\8^2+\frac{1}{(1-\kappa)^2}\sup_{0\le t\le T}|Y^\ep(t)|^2,\\
&\sup_{-\tau\le t\le T}|X^\ep(t)|^2\le\frac{1}{1-\kappa}\|\xi\|_\8^2+\frac{1}{(1-\kappa)^2}\sup_{0\le t\le T}|Y^\ep(t)|^2.
\end{split}
\end{equation}
For $\lambda>0$, applying the It\^o formula, ({\bf H1}),({\bf H2}) and \eqref{eqr1} yield
\begin{equation}\label{eq1.17}
\begin{split}
(1+|Y^\ep(t)|^2)^\lambda&\le(1+(1+\kappa)^2\|\xi\|_\8^2)^\lambda+\lambda\int_0^t(1+|Y^\ep(s)|^2)^{\lambda-1}2\langle Y^\ep(s),b(X_s^\ep)\rangle\d s\\
&~~+2\lambda(\lambda-1)\ep\int_0^t(1+|Y^\ep(s)|^2)^{\lambda-2}|\sigma(X_s^\ep)Y^\ep(s)|^2\d s\\
&~~+\lambda\ep\int_0^t(1+|Y^\ep(s)|^2)^{\lambda-1}\|\sigma(X_s^\ep)\|_{HS}^2\d s+M^{\ep,\lambda}(t)\\
&\le(1+(1+\kappa)^2\|\xi\|_\8^2)^\lambda+M^{\ep,\lambda}(t)\\
&~~+\lambda L_2(1+2\lambda\ep-\ep)\int_0^t(1+|Y^\ep(s)|^2)^{\lambda-1}(1+\|X_s^\ep\|_\8^2)\d s\\
&\le(1+(1+\kappa)^2\|\xi\|_\8^2)^\lambda+M^{\ep,\lambda}(t)\\
&~~+\lambda L_2C_1(1+2\lambda\ep-\ep)\int_0^t\Big(\sup_{0\le u\le s}(1+|Y^\ep(u)|^2)^{\lambda}\Big)\d s,
\end{split}
\end{equation}
where $C_1=(1+\frac{\|\xi\|_\8^2}{(1-\kappa)})\vee(\frac{1}{(1-\kappa)^2})$,
$M^{\ep,\lambda}(t)=2\lambda\ep\int_0^t(1+|Y^\ep(s)|^2)^{\lambda-1}\langle Y^\ep(s),\sigma(X_s^\ep)\d W(s)\rangle$, and in the last step, we used \eqref{eqv}.

Noting that $\|X_s^\ep\|_\8^2\le \|\xi\|_\8^2+\Big(\sup_{0\le u\le s}|X^\ep(u)|^2\Big)$,
by ({\bf H1}), \eqref{eqv} and the BDG inequality, we obtain
\begin{equation}\label{eq1.19}
\begin{split}
&\E\Big(\sup_{0\le t\le T}M^{\ep,\lambda}(t)\Big)\\
&\le8\ss{2\ep}\lambda\bigg(\E\int_0^T(1+|Y^\ep(s)|^2)^{2\lambda-1}\|\sigma(X_s^\ep)\|_{HS}^2\d s\bigg)^{1/2}\\
&\le\frac{1}{2}\E\Big(\sup_{0\le t\le T}(1+|Y^\ep(s)|^2)^\lambda\Big)+64L_2\lambda^2\ep\E\int_0^T(1+|Y^\ep(s)|^2)^{\lambda-1}(1+\|X_s^\ep\|_\8^2)\d s\\
&\le\frac{1}{2}\E\Big(\sup_{0\le t\le T}(1+|Y^\ep(s)|^2)^\lambda\Big)+64L_2\lambda^2C_1\ep\E\int_0^T\Big(\sup_{0\le u\le s}(1+|Y^\ep(u)|^2)^{\lambda}\Big)\d s.
\end{split}
\end{equation}

Substituting \eqref{eq1.19} into \eqref{eq1.17}, and reformulating \eqref{eq1.17}, we arrive at 
\begin{equation}
\begin{split}
&\E\Big(\sup_{0\le t\le T}(1+|Y^\ep(t)|^2)^\lambda\Big)\\
&\le2(1+(1+\kappa)^2\|\xi\|_\8^2)^\lambda
+2L_2C_1\lambda[66\lambda\ep+1-\ep]\int_0^T\E\Big(\sup_{0\le u\le s}(1+|Y^\ep(u)|^2)^\lambda\Big)\d s.
\end{split}
\end{equation}
For $R>0$, we define $\xi_R^\ep=\inf\{t\ge0: |X^\ep(t)|>R\}$, utilising BDG's inequality yields that
\begin{equation*}
\E\Big(\sup_{0\le t\le T}(1+|Y^\ep(t\wedge \xi_R^\ep)|^2)^\lambda\Big)\le
2(1+(1+\kappa)^2\|\xi\|_\8^2)^\lambda\exp\{2L_2C_1\lambda[66\lambda\ep+1-\ep]T\},
\end{equation*}
which implies that
\begin{equation*}
\E\Big\{\Big(\sup_{0\le t\le T}(1+|Y^\ep(t\wedge \xi_R^\ep)|^2)^\lambda\Big)I_{\{\xi_R^\ep\le T\}}\Big\}\le
2(1+(1+\kappa)^2\|\xi\|_\8^2)^\lambda\exp\{2L_2C_1\lambda[66\lambda\ep+1-\ep]T\},
\end{equation*}
\begin{equation*}
\P\Big(\sup_{-\tau\le t\le T}|X^\ep(t)|>R\Big)\le \P(\xi_R^\ep\le T)\le 
\frac{2(1+(1+\kappa)^2\|\xi\|_\8^2)^\lambda\exp\{2L_2C_1\lambda[66\lambda\ep+1-\ep]T\}}
{\Big(1+[R-\frac{\kappa}{1-\kappa}\|\xi\|_\8^2](1-\kappa)^2\Big)^\lambda},
\end{equation*}
choosing $\lambda=\frac{1}{\ep}$ yields that
\begin{equation*}
\begin{split}
\ep\log \P\Big(\sup_{-\tau\le t\le T}|X^\ep(t)|>R\Big)
&\le\log\frac{2(1+(1+\kappa)^2\|\xi\|_\8^2)}
{\Big(1+[R-\frac{\kappa}{1-\kappa}\|\xi\|_\8^2](1-\kappa)^2\Big)}
+\ep2L_2C_1\lambda[66\lambda\ep+1-\ep]T\\
&\le\log\frac{2(1+(1+\kappa)^2\|\xi\|_\8^2)}
{\Big(1+[R-\frac{\kappa}{1-\kappa}\|\xi\|_\8^2](1-\kappa)^2\Big)}
+2L_2C_1(67-\ep)T,
\end{split}
\end{equation*}
\begin{equation*}
\lim_{R\rightarrow\8}\limsup_{\ep\rightarrow0}\ep\log P\Big(\sup_{-\tau\le t\le T}|X^\ep(t)|>R\Big)=-\8.
\end{equation*}
The proof is therefore complete.
\end{proof}

For $R>0$, define $m_R=\sup\{|G(x)|,|b(x)|,\|\sigma(x)\|_{HS}; \|x\|_\8\le R\}$, and $G_i^R=(-m_R-1)\vee G_i\wedge(m_R+1)$, $b_i^R=(-m_R-1)\vee b_i\wedge(m_R+1)$,
$\sigma_{i,j}^R=(-m_R-1)\vee \sigma_{i,j}\wedge(m_R+1)$, $1\le i,j\le d$. Let $G_R=(G_1^R,G_2^R,\cdots,G_d^R)$, $b_R=(b_1^R,b_2^R,\cdots,b_d^R)$
and $\sigma_R=(\sigma_{i,j}^R)_{1\le i,j\le d}$.
Then for $\|x\|_\8\le R$,
 $$G_R(x)=G(x),~~~~b_R(x)=b(x), ~~~~\sigma_R(x)=\sigma(x).$$

 Also, $G_R$, $b_R$ and $\sigma_R$ satisfy the assumptions ({\bf H1}) and ({\bf H2}).
Let $X^{\ep,R}(\cdot)$ be the solution to the NSFDE
\begin{equation*}
\d\{X^{\ep,R}(t)-G(X_t^{\ep,R})\}=b_R(X_t^{\ep, R})\d t+\ss\ep\sigma_R(X_t^{\ep,R})\d W(t), t>0,
\end{equation*}
with the initial datum $X_0^{\ep,R}=\xi(\theta),~~\theta\in[-\tau,0]$.

We recall a Lemma in \cite{DZ}, which is a key point in the proofs of following Lemmas.
\begin{lem}\label{lema}
Let $N$ be a fixed integer. Then, for any $a_\ep^i\ge0$,
\begin{equation}\label{eq}
\limsup_{\ep\rightarrow 0}\ep\log\Big(\sum_{i=1}^Na_\ep^i\Big)=\max_{i=1}^N\limsup_{\ep\rightarrow0}\ep\log a_\ep^i.
\end{equation}
\end{lem}
The lemma below states that $X^{\ep,R}(\cdot)$ is the uniformly exponential approximation of $X^\ep(\cdot)$ on the interval $[-\tau, T]$.
\begin{lem}\label{leml}
Assume ({\bf H1}), ({\bf H2}) hold, then for any $T>0,~~\delta>0$, one has that:
\begin{equation}\label{eq1.20}
\lim_{R\rightarrow\8}\limsup_{\ep\rightarrow 0}\ep\log P
\Big(\sup_{-\tau\le t\le T}|X^\ep(t)-X^{\ep,R}(t)|>\delta\Big)=-\8.
\end{equation}
\end{lem}
\begin{proof}
For notation simplicity, we set $Z^{\ep,R}(t):=X^\ep(t)-X^{\ep,R}(t)$ and $Y^{\ep,R}(t):=X^\ep(t)-X^{\ep,R}(t)-(G(X_t^\ep)-G(X_t^{\ep,R}))$.

From ({\bf H2}), it is easy to see that
\begin{equation*}
\sup_{0\le t\le T}|Z^{\ep,R}(t)|\le\sup_{0\le t\le T}\Big(\frac{1}{1-\kappa}|Y^{\ep,R}(t)|\Big).
\end{equation*}
Define $\xi_{R_1}^\ep:=\inf\{t\ge0:|X^\ep(t)|\ge R_1\}$. For any $R\ge R_1$, we have
\begin{equation}\label{eq1.21}
\begin{split}
Y^{\ep,R}(t\wedge\xi_{R_1}^\ep)=\int_0^{t\wedge\xi_{R_1}^\ep}(b_R(X_s^\ep)-b_R(X_s^{\ep,R}))\d s
+\ss\ep\int_0^{t\wedge\xi_{R_1}^\ep}(\sigma_R(X_s^\ep)-\sigma_R(X_s^{\ep,R}))\d W(s).
\end{split}
\end{equation}
Setting $Z_{R_1}^{\ep}(t):=Z^{\ep,R}(t\wedge\xi_{R_1}^\ep)$,
$Y_{R_1}^{\ep}(t):=Y^{\ep,R}(t\wedge\xi_{R_1}^\ep)$
and $\xi_{R,\delta}^\ep:=\inf\{t\ge 0:|Z_{R_1}^{\ep}(t)|\ge\delta\}$.
Then, we have
\begin{equation}\label{eq1.22}
\begin{split}
&P\Big(\sup_{-\tau\le t\le T}|Z^{\ep,R}(t)|>\delta\Big)\\
&=P\Big(\sup_{-\tau\le t\le T}|Z^{\ep,R}(t\wedge\xi_{R_1}^\ep)|>\delta, I_{\{\xi_{R_1}^\ep\ge T\}}\Big)
+P\Big(\sup_{-\tau\le t\le T}|Z^{\ep,R}(t\wedge\xi_{R_1}^\ep)|>\delta, I_{\{\xi_{R_1}^\ep\le T\}}\Big)\\
&\le P(\xi_{R_1}^\ep\le T)+P(\xi_{R,\delta}^\ep\le T)\\
&\le P\Big(\sup_{-\tau\le t\le T}|X^\ep(t)|>R_1\Big)+P(\xi_{R,\delta}^\ep\le T).
\end{split}
\end{equation}
By mimicking the argument in Lemma \ref{lem3} for $t\le T\wedge\xi_{R_1}^\ep$, one gets
\begin{equation*}
\mathbb{E}\Big(\sup_{0\le t\le T}\big(\rho^2+|Y_{R_1}^\ep(t)|^2\big)^{1/\ep}\Big)\le 2\rho^{2/\ep}\e^{CT/\ep}.
\end{equation*}
This implies that
\begin{equation*}
P(\xi_{R,\delta}^\ep\le T)\le\Big(\frac{2^\ep\rho^2}{\rho^2+(1-\kappa)^2\delta^2}\Big)^{1/\ep}\e^{CT/\ep}.
\end{equation*}
Taking Logarithmic function into consideration, we have
\begin{equation*}
\limsup_{\ep\rightarrow 0}\ep\log P(\xi_{R,\delta}^\ep\le T)\le\log\Big(\frac{\rho^2}{\rho^2+(1-\kappa)^2\delta^2}\Big)+CT.
\end{equation*}
This, together with \eqref{eq1.16},\eqref{eq} and \eqref{eq1.22},  implies
\begin{equation*}
\begin{split}
&\lim_{R\rightarrow\8}\limsup_{\ep\rightarrow 0}\ep\log P\Big(\sup_{-\tau\le t\le T}|Z^{\ep,R}(t)|>\delta\Big)\\
&\le\lim_{R\rightarrow\8}\limsup_{\ep\rightarrow 0}\ep\log \Big(P\Big(\sup_{-\tau\le t\le T}|X^\ep(t)|>R_1\Big)+\lim_{R\rightarrow\8}\limsup_{\ep\rightarrow 0}P(\xi_{R,\delta}^\ep\le T)\Big)\\
&\le\limsup_{\ep\rightarrow 0}\ep\log P\Big(\sup_{-\tau\le t\le T}|X^\ep(t)|>R_1\Big)\vee\Big\{\log\Big(\frac{\rho^2}{\rho^2+(1-\kappa)^2\delta^2}\Big)+CT\Big\}.
\end{split}
\end{equation*}
The conclusion follows from letting first $\rho\rightarrow 0$ and then $R_1\rightarrow\8$ by Lemma \ref{lemr}.
\end{proof}
For $h$ with $L_T(h)<\8$, let $F^R(h)$ be the solution of the  equation below
\begin{equation*}
F^R(h)(t)-G(F_t^R(h))=F^R(h)(0)-G(F_0^R(h))+\int_0^tb_R(F_s^R(h))\d s+\int_0^t\sigma_R(F_s^R(h))\dot{h}(s)\d s
\end{equation*}
with the initial datum $F_0^R(h)=\xi(\theta),~~\theta\in[-\tau,0]$.
Define
\begin{equation*}
I_R(f)=\inf\Big\{\frac{1}{2}\int_0^T|\dot{h}(t)|^2\d t;~~F^R(h)=f\Big\},
\end{equation*}
for each $f\in C([-\tau,T];\R^d)$. If $\Big(\sup_{-\tau\le t\le T}|F(h)(t)|\Big)\le R$, then $F(h)=F^R(h)$.
$$I(f)=I_R(f), ~~\mbox{for all}~ f~\mbox{with} ~~\Big(\sup_{-\tau\le t\le T}|f(t)|\Big)\le R.$$
{\bf Proof of Theorem \ref{th1} in case 2}
\begin{proof}
For $R>0$, and a closed subset $C\subset C([-\tau,T];\R^d)$, set $C_R:=C\cap\{f;\|f\|_\8\le R\}$.
$C_R^\delta$ denotes the $\delta$-neighborhood of $C_R$. Denote by $\mu^{\ep,R}$ the law of $X_R^\ep$.
Then we have
\begin{equation*}
\begin{split}
\mu_\ep(C)&=\mu_\ep(C_{R_1})+\mu_\ep\Big(C,\sup_{-\tau\le t\le T}|X^\ep(t)|>R_1\Big)\\
&\le\mu_\ep(C_{R_1})+P\Big(\sup_{-\tau\le t\le T}|X^\ep(t)|>R_1\Big)\\
&\le P\Big(\sup_{-\tau\le t\le T}|X^\ep(t)-X^{\ep,R}(t)|>\delta\Big)
+\mu_\ep^R\Big(C_{R_1}^\delta\Big)
+P\Big(\sup_{-\tau\le t\le T}|X^\ep(t)|>R_1\Big).
\end{split}
\end{equation*}
Taking the large deviation principle for $\{\mu_\ep^R,\ep>0\}$ yields from \ref{lema} that
\begin{equation*}
\begin{split}
&\limsup_{\ep\rightarrow0}\ep\log\mu_\ep(C)\\
&\le\limsup_{\ep\rightarrow0}\ep\log
\Big\{P\Big(\sup_{-\tau\le t\le T}|X^\ep(t)-X^{\ep,R}(t)|>\delta\Big)
+\Big(-\inf_{f\in C_{R_1}^\delta}I_R(f)\Big)\\
&~~+P\Big(\sup_{-\tau\le t\le T}|X^\ep(t)|>R_1\Big)
\Big\}\\
&\le\Big(-\inf_{f\in C_{R_1}^\delta}I_R(f)\Big)
\vee\Big(\limsup_{\ep\rightarrow 0}\ep\log P\big(\sup_{-\tau\le t\le T}|X^\ep(t)|>R_1\big)\Big)\\
&~~~~\vee\Big(\limsup_{\ep\rightarrow0}\ep\log P\big(\sup_{-\tau\le t\le T}|X^\ep(t)-X^{\ep,R}(t)|>\delta\big)\Big).
\end{split}
\end{equation*}
Then we obtain the upper bound (i) in Theorem \ref{th1}, that is
\begin{equation*}
\limsup_{\ep\rightarrow 0}\ep\log\mu_\ep(C)\le -\inf_{f\in C}I(f),
\end{equation*}
by taking first $R\rightarrow\8$, and $\delta\rightarrow0$, then $R_1\rightarrow\8$.
Let $G$ be an open subset of $C([-\tau,T];\R^d)$.  Then for any $\phi_0\in G$, and taking $\delta>0$, we define
$B(\phi_0,\delta)=\{f;\|f-\phi_0\|_\8\le\delta\}\subset G$.
Then using the large deviation principle for $\{\mu_\ep^R;\ep>0\}$, one gets
\begin{equation*}
\begin{split}
-I_R(\phi_0)&\le\liminf_{\ep\rightarrow0}\ep\log\mu_\ep^R\Big(B(\phi_0,\frac{\delta}{2})\Big)\\
&\le\liminf_{\ep\rightarrow0}\ep\log\Big\{ P\Big(\sup_{-\tau\le t\le T}|X^{\ep,R}(t)-\phi_0|\le\frac{\delta}{2},
\sup_{-\tau\le t\le T}|X^\ep(t)-X^{\ep,R}(t)|\le\frac{\delta}{2}\Big)\\
&~~~~~~~~~~~~~~~~~~~~~+ P\Big(\sup_{-\tau\le t\le T}|X^{\ep,R}(t)-\phi_0|\le\frac{\delta}{2},
\sup_{-\tau\le t\le T}|X^\ep(t)-X^{\ep,R}(t)|\ge\frac{\delta}{2}\Big)\Big\}\\
&\le\Big(\liminf_{\ep\rightarrow0}\ep\log\mu_\ep(G)\Big)\vee\Big(\liminf_{\ep\rightarrow0}\ep\log P\big(\sup_{-\tau\le t\le T}|X^\ep(t)-X^{\ep,R}(t)|\ge\frac{\delta}{2}\big)\Big).
\end{split}
\end{equation*}
Noting that $I_R(\phi_0)=I(\phi_0)$ provided that $\|\phi_0\|_\8\le R$. Then we have
\begin{equation*}
-I(\phi_0)\le\liminf_{\ep\rightarrow0}\ep\log\mu_\ep(G), ~\mbox{as}~~R\rightarrow\8.
\end{equation*}
Owing to the arbitrary of $\phi_0$, it follows that
\begin{equation*}
-\inf_{f\in G}I(f)\le\liminf_{\ep\rightarrow0}\ep\log\mu_\ep(G),
\end{equation*}
which is the lower bound (i) in Theorem \ref{th1}, thus, the proof of Theorem \ref{th1} is complete.
\end{proof}

\end{document}